\documentclass[11pt, a4paper]{amsart}

\usepackage[T1]{fontenc}
\usepackage{microtype}
\usepackage{amsmath}								%math
\usepackage{amssymb}
\usepackage{amsthm}
\usepackage{amscd}
\usepackage{amsfonts}
\usepackage{mathabx}
\usepackage{mathtools}
\usepackage{accents}
\usepackage{comment}
\usepackage{stmaryrd}

\usepackage{extarrows}
\usepackage[bookmarks=true, colorlinks=true, linkcolor=blue!50!black,
citecolor=orange!50!black, urlcolor=orange!50!black, pdfencoding=unicode]{hyperref}
\usepackage{color}

\usepackage{tikz}									
\usetikzlibrary{matrix}
\usetikzlibrary{patterns}
\usetikzlibrary{matrix}
\usetikzlibrary{positioning}
\usetikzlibrary{decorations.pathmorphing}
\usetikzlibrary{cd}

\usepackage{nicefrac}
\usepackage{fullpage}
\usepackage{enumitem}

\newtheorem*{theorem*}{Theorem}

\newtheorem{maintheorem}{Theorem}[section]
\newtheorem{theorem}{Theorem}[section]

\newtheorem{proposition}[theorem]{Proposition}

\newtheorem{question}[theorem]{Question}

\theoremstyle{definition}
\newtheorem{definition}[theorem]{Definition}
\newtheorem*{definition*}{Definition}
\newtheorem{example}[theorem]{Example}

\newtheorem{remark}[theorem]{Remark}

\newtheoremstyle{myitemstyle}						% flexible theorem style
	{}			%Space above
	{}			%Space below
	{}			%Body font
	{}			%indent amount
	{}			%Thm head font
	{.}			%Punkte nach thm head
	{ }			%Abstand nach thm head
	{}			%Thm head spec
\theoremstyle{myitemstyle}
\newtheorem{myitemthm}{}

\newcommand{\R}{\mathbb{R}}
\newcommand{\Rbar}{\overline{\mathbb{R}}}

\newcommand{\Z}{\mathbb{Z}}

\newcommand{\calB}{\mathcal{B}}

\newcommand{\sfA}{\mathsf{A}}

\newcommand{\sfM}{\mathsf{M}}

\newcommand{\sfP}{\mathsf{P}}

\DeclareMathOperator{\val}{val}

\DeclareMathOperator{\St}{\mathsf{St}}

\DeclareMathOperator{\Hess}{Hess}

\newcommand{\trop}{{\mathrm{trop}}}

\DeclareMathAlphabet{\mymathbb}{U}{bbold}{m}{n}

\newcommand{\bigmid}{\mathrel{\big|}}

\def\multiset#1#2{\ensuremath{\left(\kern-.3em\left(\genfrac{}{}{0pt}{}{#1}{#2}\right)\kern-.3em\right)}}

\makeatletter
\newcommand{\superimpose}[2]{{\ooalign{$#1\@firstoftwo#2$\cr\hfil$#1\@secondoftwo#2$\hfil\cr}}}
\makeatother
\title[Log-concavity for independent sets of valuated matroids]{Log-concavity for independent sets of valuated matroids} 

\date{}

\author[J.\ Giansiracusa]{Jeffrey Giansiracusa}
\address{Department of Mathematical Sciences, Durham University, United Kingdom}
\email{jeffrey.giansiracusa@durham.ac.uk}

\author[F.\ Rinc{\'o}n]{Felipe Rinc{\'o}n}
\address{School of Mathematical Sciences, Queen Mary University of London, United
Kingdom.}
\email{f.rincon@qmul.ac.uk}

\author[V.\ Schleis]{Victoria Schleis}
\address{Department of Mathematical Sciences, Durham University, United Kingdom}
\email{victoria.m.schleis@durham.ac.uk}

\author[M.\ Ulirsch]{Martin Ulirsch}
\address{Institut f\"ur Mathematik, Goethe--Universit\"at Frankfurt,
Germany}
\email{ulirsch@math.uni-frankfurt.de}

\begin{document}

\begin{abstract} 
Recently, several proofs of the Mason--Welsh conjecture for matroids have been found, which asserts the log-concavity of the sequence that counts independent sets of a given size. In this article we use the theory of Lorentzian polynomials, developed by Br\"and\'en and Huh, to prove a generalization of the Mason-Welsh conjecture to the context of valuated matroids. In fact, we provide a log-concavity result in the more general setting of valuated discrete polymatroids, or equivalently, M-convex functions. Our approach is via the construction of a generic extension of a valuated matroid or M-convex function, so that the bases of the extension are related to the independent sets of the original matroid. We also provide a similar log-concavity result for valuated bimatroids, which, we believe, might be of independent interest.
\end{abstract}

\maketitle

%\setcounter{tocdepth}{1}
%\tableofcontents

%%%%%%%%%%%%%%%%%%%%%%%%%%%%%%%%%%%%%%%%%%%%%%%%%%%%%%

\section*{Introduction}

In \cite{DressWenzel}, Dress and Wenzel introduced \emph{valuated matroids} building a bridge between combinatorics and geometry over non-Archimedean fields. With the rise of tropical geometry in the last decade(s) the interest in valuated matroids has seen a significant increase, since they are  essentially the same thing as \emph{tropical linear spaces} (see e.g.\ \cite[Chapter 4]{MaclaganSturmfels} and \cite[Section 10]{Joswig_book} for details on this rich story). 

Let $E$ be a finite set and $\Gamma$ a totally ordered additive abelian group.  
Given an integer $r\geq 0$, we write $\binom{E}{r}$ and $\binom{E}{\leq r}$ for the set of subsets of $E$ with exactly $r$ or at most $r$ elements, respectively. We also write $\overline{\Gamma}=\Gamma\sqcup\{\infty\}$.

A \textbf{valuated matroid $\sfM$} of rank $r\geq 0$ on $E$ is given by a map $\nu_{\sfM}\colon {E\choose r}\rightarrow\overline{\Gamma}$ that is not everywhere equal to $\infty$ and fulfils the following valuated enrichment of the symmetric basis exchange property: 
Given two subsets $S,T\in{E\choose r}$ as well as $s\in S-T$ there is a $t\in T-S$ such that 
\begin{equation}\label{eq_Pluecker}
\nu_{\sfM}(S)+\nu_{\sfM}(T) \ \ \geq \ \ \nu_{\sfM}\big(S-\{s\}\cup\{t\}\big)+\nu_{\sfM}\big(T-\{t\}\cup\{s\}\big) \ .
\end{equation}
For a valuated matroid $\sfM$ of rank $r$ on $E$, we note that the set
\begin{equation*}
\calB(\sfM)=\big\{B\in{E\choose r}\mid \nu_{\sfM}(B)\neq \infty\big\}
\end{equation*}
is the set of bases of a matroid called the \textbf{underlying matroid}. In fact, when $\Gamma=\{0\}$, the datum of a valuated matroid with values in $\overline{\Gamma}$ is nothing but a matroid. 

\subsection*{A valuated version of the Mason--Welsh conjecture} In \cite{Murota_valuationindependentsets}, Murota uses a natural extension of $\nu_{\sfM}$ to a map $\widetilde{\nu}_\sfM:{E\choose\leq r}\rightarrow \overline{\Gamma}$ in order to establish a cryptomorphic characterization of valuated matroids that expands on the characterization of (non-valuated) matroids in terms of independent sets. 
The extension $\widetilde{\nu}_\sfM$ is given by 
\begin{equation}\label{eq_extensionleqr}
\widetilde{\nu}_\sfM(S):=\min\big\{\nu_{\sfM}(B)\mid B\in{E\choose r} \textrm{ and } S\subseteq B\big\} 
\end{equation} 
for $S\subseteq E$ with $\vert S\vert \leq r$. 

The following theorem  
generalizes the Mason--Welsh log-concavity conjecture for independent sets of matroids to the realm of valuated matroids. 
From now on, we assume that $\Gamma$ is an additive subgroup of $\R$ so that exponentiation makes sense.

\begin{maintheorem}\label{mainthm_valMason}
Let $\sfM=(E,\nu_{\sfM})$ be a valuated matroid of rank $r$ on a finite ground set $E$ and denote by $\widetilde{\nu}_\sfM$ the extension of $\nu_{\sfM}$ to all of ${E\choose \leq r}$ as above. Given a constant $0<q\leq 1$, we define
\begin{equation*}
I_k(\sfM):=\sum_{S\in{E\choose k}} q^{\widetilde{\nu}_\sfM(S)} 
\end{equation*}
for $0\leq k\leq r$, where by convention $q^\infty = 0$. 
Then we have
\begin{equation*}
I_k(\sfM)^2 \ \ \geq \ \ \frac{k+1}{k}\cdot I_{k+1}(\sfM)\cdot I_{k-1}(\sfM)  
\end{equation*}
for all $1\leq k<r$. In particular, the sequence $I_k(\sfM)$ is log-concave. 
\end{maintheorem}

Note that when $q=1$,  
the number $I_k(\sfM)$ 
is the number of independent sets of the underlying matroid of $\sfM$. In this case, Theorem \ref{mainthm_valMason} reduces to a 
stronger version of Mason--Welsh's original log-concavity conjecture for (non-valuated) matroids.

A first proof of the Mason--Welsh conjecture in the non-valuated situation has appeared in \cite{AHK} using the setup provided in \cite{Lenz_logconcavefvector}. In \cite{BrandenHuh}, the authors use Lorentzian polynomials to prove a strengthening of this inequality, commonly referred to as \textbf{ultra log-concavity} (also see \cite{AnariLiuGharanVinzantIII} for a similar approach via log-concave polynomials). Let $N=\vert E\vert$. Then this means that the inequality 
\begin{equation*}
\Bigg(\frac{I_k(\sfM)}{{N \choose k}}\Bigg)^2 \ \ \geq \ \ \frac{I_{k+1}(\sfM)}{{N \choose k+1}}\cdot \frac{I_{k-1}(\sfM)}{{N\choose k-1}}
\end{equation*}
holds for all $1\leq k<N$. We refer to \cite{HuhSchroeterWang} and \cite{ChanPak} for other approaches to this story.  

In Section \ref{section_ultralogconcavity} below we conjecture, based on computational evidence, that a strengthening of Theorem \ref{mainthm_valMason} to ultra log-concavity in the setting of valuated matroids might hold.

Our proof of Theorem \ref{mainthm_valMason} makes crucial use of the theory of Lorentzian polynomials. 
This approach is similar to the approach to the Mason--Welsh conjecture presented in \cite{Lenz_logconcavefvector} in the representable case, which was generalized to the non-representable case in \cite{AHK}. It also expands on the recent perspective on the Mason--Welsh conjecture established in \cite{RoehrleUlirsch}. 

The central novel ingredient in our proof is the construction of a suitable `generic' extension $\widetilde{\sfM}$ of a valuated matroid $\sfM$ to a larger ground set such that the basis valuations of $\widetilde{\sfM}$ are equal to the independent set valuations of $\sfM$ (see Proposition \ref{prop_genericextension} below). We refer the reader to \cite{FinkOlarte, BrandenburgLohoSmith} for more insights towards a general theory of extensions and quotients of valuated matroids, and to \cite[Section 7.2 and 7.3]{Oxley} for the classical non-valuated story of extensions of matroids.

\subsection*{Log-concavity of valuated polymatroids} Our approach to the proof of Theorem \ref{mainthm_valMason} via the construction of generic extensions of valuated matroids can be generalized to the setting of valuated polymatroids.

Given a finite set $E$, we write $e_s$ for $s\in E$ the standard basis vector of $\Z^E$, whose $s$-th entry is one and is zero elsewhere, and 
\begin{equation*}
\Delta_E^k=\big\{\alpha\in\Z_{\geq 0}^E\bigmid \vert \alpha\vert:=\sum_{e\in E}\alpha_e=k\big\}
\end{equation*}
for the \textbf{$k$-th discrete simplex} on $E$. We may interpret $\Delta_E^k$ as the set  $\multiset{E}{k}$ of multisets of size $k$ in $E$ by associating to $\alpha\in\Delta_E^k$ the multiset consisting of $\alpha_e$ many copies of the element $e\in E$. Under this correspondence, the intersection of $\Delta_E^k$ with the hypercube $\{0,1\}^E$ corresponds precisely to ${E\choose r}$. 

Discrete polymatroids are a multiset generalization of matroids that abstracts the behaviour of subspaces of a fixed vector space.  
A \textbf{valuated (discrete) polymatroid} $\sfP$ of rank $r$ on $E$ can be defined in terms of its associated $M$-convex function (as introduced in \cite{Murota_Convexity&Steinitz,Murota_valuatedmatroidintersectionI,Murota_valuatedmatroidintersectionII}). 
An \textbf{$M$-convex function} of rank $r$ is a function $\nu_{\sfP}\colon\Delta_E^r\rightarrow \overline{\Gamma}$, not equal to the constant function $\infty$, that fulfills a generalization of the valuated Pl\"ucker relation \eqref{eq_Pluecker} above: For all $\alpha,\beta\in\Delta_E^r$ and $s\in E$ such that $\alpha_s>\beta_s$ there is $t\in E$ with $\beta_t>\alpha_t$ such that 
\begin{equation}\label{eq_MconvexPluecker}
    \nu_{\sfP}(\alpha)+\nu_{\sfP}(\beta) \ \ \geq \ \ \nu_{\sfP}(\alpha-e_s+e_t) + \nu_{\sfP}(\beta-e_t+e_s) \ . 
\end{equation}

We write 
\begin{equation*}
    \Delta_E^{\leq r}:=\big\{\alpha\in\Z_{\geq 0}^E\bigmid\vert\alpha\vert\leq r\big\}.
\end{equation*}
Analogous to Murota's extension of a basis valuation function to independent sets, we define a function $\widetilde{\nu}_\sfP\colon \Delta_{E}^{\leq r}\rightarrow\overline{\Gamma}$ that extends $\nu_\sfP$ by setting
\begin{equation*}
\widetilde{\nu}_\sfP(\alpha):=\min\big\{\nu_{\sfP}(\beta)\mid \beta\in\Delta_{E}^{r} \textrm{ and } \alpha_e\leq \beta_e \textrm{ for all }e\in E\big\} 
\end{equation*} 
for $\alpha\in \Delta_{E}^{\leq r}$. Given $\alpha\in\Z^E_{\geq 0}$, we write $\alpha!=\prod_{e\in E}\alpha_e!$.  

The following theorem generalizes Theorem \ref{mainthm_valMason} to the context of polymatroids.

\begin{maintheorem}\label{mainthm_logconcavitypolymatroids}
Let $\sfP$ be a valuated polymatroid of rank $r$ on a finite set $E$. Given $0<q\leq 1$, we define
\begin{equation*}
I_k(\sfP):=\sum_{\alpha\in\Delta_E^k}\frac{q^{\widetilde{\nu}_\sfP(\alpha)}}{\alpha!}
\end{equation*}
for every $0\leq k\leq r$, where by convention $q^\infty = 0$. Then we have 
\begin{equation*}
    I_k(\sfP)^2 \ \ \geq \ \ \frac{k+1}{k}\cdot I_{k+1}(\sfP)\cdot I_{k-1}(\sfP)
\end{equation*}
for all $1\leq k<r$. In particular, the sequence $I_k(\sfP)$ is log-concave.
\end{maintheorem}

We point out that even in the non-valuated case of discrete polymatroids, i.e., when $\nu_{\sfP}$ only takes values $0$ or $\infty$, Theorem \ref{mainthm_logconcavitypolymatroids} appears to be a new result. 

Theorem \ref{mainthm_logconcavitypolymatroids} is proved via the construction of a generic extension of $M$-convex functions in Proposition \ref{prop_genericextensionMconvexfunction} below. The construction uses extensions of valuated matroids (Proposition \ref{prop_genericextension}) and the characterization of $M$-convex functions in terms of their multi-symmetric lifts in Proposition \ref{prop_multisymmetriclift} (expanding on the non-valuated case treated in \cite{EurLarson}).

\subsection*{Log-concavity of valuated bimatroids} In \cite{Kung_bimatroids}, Kung introduces the notion of a \textbf{bimatroid} as a combinatorial abstraction of the set of regular minors of a fixed matrix (also see \cite{Schrijver_linkingsystems} for the equivalent notion of \textbf{linking systems} and \cite{RoehrleUlirsch} for a more recent account). 

In \cite{Murota_valuatedbimatroids}, Murota proposes a valuated generalization of bimatroids. Given two finite disjoint sets $E$ and $F$, we write ${E\choose \ast}\times{F\choose \ast}$ for the collection of pairs $(I,J)$ consisting of $I\subseteq E$ and $J\subseteq F$ of the same cardinality.  A \textbf{valuated bimatroid} $\sfA$ is given by a map $\mu_\sfA\colon {E\choose \ast}\times{F\choose \ast}\rightarrow \overline{\Gamma}$ such that the map $\nu_{\widehat{\sfA}}:\binom{E \sqcup F}{|E|} \rightarrow \overline{\Gamma}$ given by
\begin{equation*}
    \nu_{\widehat{\sfA}}(S) = \mu_\sfA(E-S,S\cap F)
\end{equation*}
is a valuated matroid $\widehat{\sfA}$ of rank $|E|$ on the disjoint union $E\sqcup F$ satisfying $\nu_{\widehat{\sfA}}(E)=0$. Given a matrix $A\in K^{E\times F}$ over a non-Archimedean valued field $K$, the function $\mu_A$ that is given by taking (negative) valuations of all minors of $A$ forms a \textbf{realizable valuated bimatroid} (see Example \ref{example_tropicalization} below). 

The following theorem is a valuated generalization of \cite[Theorem A]{RoehrleUlirsch}.

\begin{maintheorem}\label{mainthm_logconcavitybimatroids}
    Let $E$ and $F$ be finite sets and $\sfA$ be a valuated bimatroid on the rows $E$ and columns $F$ 
    with values in $\overline{\Gamma}\subseteq\Rbar$. Set 
    \begin{equation}\label{eq_R_k}
        R_k(\sfA):=\sum_{(I,J)\in{E\choose k}\times{F\choose k}}q^{\mu_{\sfA}(I,J)}  
    \end{equation}
    for $k\geq 0$, where by convention $q^\infty = 0$. Then the sequence $R_k$ is ultra log-concave, i.e. we have 
    \begin{equation*}
        \frac{R_k(\sfA)^2}{{N\choose k}^2}\ \ \geq \ \ \frac{R_{k+1}(\sfA)}{{N\choose k+1}}\cdot \frac{R_{k-1}(\sfA)}{{N\choose k-1}}\ ,
    \end{equation*}
    for every $N\geq\min\big\{\vert E\vert,\vert F\vert\big\}$.
\end{maintheorem}

When $q=1$, the number $R_k(\sfM)$ is the number of regular  minors of $\sfA$. So Theorem \ref{mainthm_logconcavitybimatroids} recovers the ultra log-concavity of this sequence, which appeared in \cite[Theorem A]{RoehrleUlirsch}. If $\sfA$ is realizable by a matrix $A\in K^{E\times F}$, then Theorem \ref{mainthm_logconcavitybimatroids} tells us that the sequence $R_k(A)$ defined by
\begin{equation*}
    R_k(A):=\sum_{(I,J)\in{E\choose k}\times{F\choose k}}q^{\val(\det[A]_{I,J})} 
\end{equation*}
as in \eqref{eq_R_k} is ultra log-concave.

Let $\sfM$ be a valuated matroid on a finite set $F$. We note that the inequality $I_k(\sfM)^2\geq I_{k+1}(\sfM)\cdot I_{k-1}(\sfM)$ implied by Theorem \ref{mainthm_valMason} also follows from Theorem \ref{mainthm_logconcavitybimatroids}. Given a finite set $E$ disjoint from $F$, in Proposition \ref{prop_genericextension} below we construct a valuated matroid $\widetilde{\sfM}$ of rank $\vert E\vert$ on $E\sqcup F$ satisfying $\nu_{\widetilde{\sfM}}(E)=0$ and $\nu_{\widetilde{\sfM}}(\widetilde{S})=\widetilde{\nu}_\sfM(\widetilde{S}\cap F)$ for $\widetilde{S}\in{E\sqcup F\choose \vert E\vert}$ with $\widetilde{S}\cap F\neq \emptyset$. By Proposition \ref{prop_extendedmatroid} below, we may interpret $\widetilde{\sfM}$ as the valuated matroid $\widehat{\sfA}$ of a valuated bimatroid $\sfA$ of type $E\times F$ and, by construction of $\widetilde{\sfM}$, we have $R_k(\sfA)=I_k(\sfM) \cdot {\vert E\vert \choose k}$. This implies the weaker inequality 
\begin{equation*}
I_k(\sfM)^2 \ \ \geq \ \  I_{k+1}(\sfM)\cdot I_{k-1}(\sfM)  
\end{equation*}
for all $1\leq k<r$ in Theorem \ref{mainthm_valMason}. 

In \cite{RoehrleUlirsch}, R\"ohrle and the fourth author use a similar bimatroid extension to reprove a log-concavity result for the number of bases of morphism of matroids that originally is a consequence of \cite[Theorem 1.3]{EurHuh}. However, a similar approach for affine morphisms of valuated matroids, as defined in \cite{IezziSchleis} and expanding on \cite{BrandtEurZhang}, does not seem to immediately yield such a result. The reason is that quotients of valuated matroids are not known to admit a natural Higgs factorization (see \cite[Section 7.3]{Oxley} for the classical story). We refer the reader to \cite{BrandenburgLohoSmith} and \cite{LorscheidJarra_flagmatroids} for further background on the intricacies of quotients of valuated matroids. 

\subsection*{Acknowledgements}
The authors would like to thank Madeline Brandt, Alex Fink, Georg Loho, Michael Joswig, Kevin K\"uhn, Arne Kuhrs, Felix R\"ohrle, Yue Ren, Pedro Souza, and Hendrik S\"u{\ss} for helpful conversations and discussions en route to this article. V.S.\ and M.U.\ have learned how to use Lorentzian polynomials during the 2023 Chow Lectures at the MPI in Leipzig, given by June Huh. We thank all speakers and organizers at this occasion. 

\subsection*{Funding} This project has received funding from the Deutsche Forschungsgemeinschaft (DFG, German Research Foundation) TRR 326 \emph{Geometry and Arithmetic of Uniformized Structures}, project number 444845124, from the DFG Sachbeihilfe \emph{From Riemann surfaces to tropical curves (and back again)}, project number 456557832, as well as the DFG Sachbeihilfe \emph{Rethinking tropical linear algebra: Buildings, bimatroids, and applications}, project number 539867663, within the  
SPP 2458 \emph{Combinatorial Synergies}.

V.S. was supported by DFG TRR 195 \emph{Symbolic Tools in Mathematics and their Application} Project-ID 286237555, and by the UKRI FLF \emph{Computational Tropical Geometry}  MR/Y003888/1.

%%%%%%%%%%%%%%%%%%%%%%%%%%%%%%%%%%%%%%%%%%%%%%%%%%%%%%

\section{Generic extensions of valuated matroids and $M$-convex functions}

In this section and the next we only suppose that $\Gamma$ is a totally ordered additve abalian group $\Gamma$ and write $\overline{\Gamma}=\Gamma\sqcup \{\infty\}$. 

While the original definition of a valuated matroid is a valuated generalization of the symmetric basis exchange property, in \cite{Murota_valuationindependentsets} Murota established a cryptomorphic characterization of valuated matroid expanding on the notion of independent sets. 
In this setup, a \textbf{valuated matroid} of rank $r$ on $E$ is given by a map $\widetilde{\nu}_\sfM\colon {E\choose \leq r}\rightarrow \overline{\Gamma}$ that fulfils the following four properties:
\begin{enumerate}
\item There is an $S\in{E\choose r}$ such that $\widetilde{\nu}_\sfM(S)\neq \infty$.
\item For $S\subseteq T$ we have $\widetilde{\nu}_\sfM(S)\leq \widetilde{\nu}_\sfM(T)$. 
\item For $S\in{E\choose \leq r-1}$ with $\widetilde{\nu}_\sfM(S)<\infty$ there is $e\in E-S$ such that $\widetilde{\nu}_\sfM(S\cup\{e\})=\widetilde{\nu}_\sfM(S)$. 
\item For $S,T\in{E\choose \leq r}$ with $\vert S\vert <\vert T\vert$ there is $t\in T-S$ such that 
\begin{equation*}
\widetilde{\nu}_\sfM(S)+\widetilde{\nu}_\sfM(T) \ \ \geq \ \ \nu_{\sfM}\big(S\cup\{t\}\big)+\nu_{\sfM}\big(T-\{t\}\big) \ .
\end{equation*}
\end{enumerate}
Given $\widetilde{\nu}_\sfM$ one can recover $\nu_{\sfM}$ by restricting it to ${E\choose r}$ and, conversely, given $\nu_{\sfM}$, we find $\widetilde{\nu}_\sfM$ by setting
\begin{equation*}
\widetilde{\nu}_\sfM(S):=\min\big\{\nu_{\sfM}(B)\mid B\in{E\choose r} \textrm{ and } S\subseteq B\big\} 
\end{equation*} 
for $S\subseteq E$ with $\vert S\vert \leq r$. 

We also recall that, by \cite[Theorem 2.1]{Murota_valuationindependentsets}, given a valuated matroid $\sfM$ of rank $r$ on a finite set $E$, for every 
$0\leq \rho\leq r$,
the restriction of $\widetilde{\nu}_\sfM$ to ${E\choose \rho}$ defines a valuated matroid $\sfM_\rho$ of rank $\rho$.

The following proposition
is central to our proof of Theorem \ref{mainthm_valMason} and \ref{mainthm_logconcavitypolymatroids} and constructs a generic extension of an arbitrary valuated matroid to a larger ground set. 

\begin{proposition}\label{prop_genericextension}
Let $\sfM=(E,\nu_{\sfM})$ be a valuated matroid of rank $r$ on a finite ground set $E$. Choose a finite set $Q$ 
that is disjoint from $E$ and write $\widetilde{E}=Q\sqcup E$. Define a function $\nu_{\widetilde{\sfM}}\colon {\widetilde{E}\choose r}\longrightarrow \overline{\Gamma}$ by setting
\begin{equation*}
\nu_{\widetilde{\sfM}}(\widetilde{S}):=\widetilde{\nu}_\sfM(\widetilde{S}\cap E) 
\end{equation*}
for $\widetilde{S}\in{\widetilde{E}\choose r}$. Then the function $\nu_{\widetilde{\sfM}}$ defines a valuated matroid $\widetilde{\sfM}$ of rank $r$ on the ground set $\widetilde{E}$. 
\end{proposition}

\begin{proof}
 It is immediate that $\nu_{\widetilde{\sfM}}$ is not infinite everywhere. 
 So, given $\widetilde{S}, \widetilde{T}\in{\widetilde{E}\choose r}$ as well as $\widetilde{s}\in \widetilde{S}-\widetilde{T}$, we need to verify that there is $\widetilde{t}\in \widetilde{T}-\widetilde{S}$ such that 
\begin{equation*}
\nu_{\widetilde{\sfM}}\big(\widetilde{S}\big)+\nu_{\widetilde{\sfM}}\big(\widetilde{T}\big)\ \ \geq\ \ \nu_{\widetilde{\sfM}}\big(\widetilde{S}-\{\widetilde{s}\}\cup\{\widetilde{t}\}\big)+\nu_{\widetilde{\sfM}}\big(\widetilde{T}-\{\widetilde{t}\}\cup\{\widetilde{s}\}\big) \ .
\end{equation*}
For this we set $Q_{\widetilde{S}}:=Q\cap \widetilde{S}$ and $Q_{\widetilde{T}}:=Q\cap \widetilde{T}$ as well as $S=\widetilde{S}\cap E$ and $T=\widetilde{T}\cap E$, so that $\widetilde{S}=Q_{\widetilde{S}}\sqcup S$ and $\widetilde{T}=Q_{\widetilde{T}}\sqcup T$.

\textbf{Case A:} Suppose that $\widetilde{s}\in Q_{\widetilde S}$. If $Q_{\widetilde{S}}\not\supseteq Q_{\widetilde{T}}$, there is $q\in Q_{\widetilde{T}}$ with $q\notin Q_{\widetilde{S}}$. Setting $\widetilde{t}=q$, we find
\begin{equation*}\begin{split}
	\nu_{\widetilde{\sfM}}\big(\widetilde{S}\big)+\nu_{\widetilde{\sfM}}\big(\widetilde{T}\big) & \ \ =  \ \ \widetilde{\nu}_{\sfM}(S) + \widetilde{\nu}_{\sfM}(T) \\
	&\ \ = \ \ \nu_{\widetilde{\sfM}}\big(\widetilde{S}-\{\widetilde{s}\}\cup\{\widetilde{t}\}\big)+\nu_{\widetilde{\sfM}}\big(\widetilde{T}-\{\widetilde{t}\}\cup\{\widetilde{s}\}\big) \ .
\end{split}\end{equation*}
On the other hand, if $Q_{\widetilde{S}}\supseteq Q_{\widetilde{T}}$ we automatically have $\vert Q_{\widetilde{T}}\vert < \vert Q_{\widetilde{S}}\vert$, since $\widetilde{s}\in Q_{\widetilde{S}}-\widetilde{T}$. This, in particular, means that $\vert S\vert <\vert T\vert$. By \cite[Theorem 3.2]{Murota_valuationindependentsets}, there is $t\in T-S$ such that 
\begin{equation*}
\widetilde{\nu}_\sfM(S)+\widetilde{\nu}_{\sfM}(T)\ \ \geq \ \  \widetilde{\nu}_\sfM(S\cup\{t\}) + \widetilde{\nu}_{\sfM}(T-\{t\}) 
\end{equation*}
and, setting $\widetilde{t}=t$, this implies
\begin{equation*}\begin{split}
\nu_{\widetilde{\sfM}}(\widetilde{S})+\nu_{\widetilde\sfM}(\widetilde{T}) \ \ & = \ \ \widetilde{\nu}_\sfM(S)+\widetilde{\nu}_{\sfM}(T)\\
 & \geq \ \  \widetilde{\nu}_\sfM(S\cup\{t\}) + \widetilde{\nu}_{\sfM}(T-\{t\})\\
 & = \ \ \nu_{\widetilde{\sfM}}(\widetilde{S}-\{\widetilde{s}\}\cup \{\widetilde{t}\})+\nu_{\widetilde\sfM}(\widetilde{T}-\{\widetilde{t}\}\cup\{\widetilde{s}\}) \ .
\end{split}\end{equation*}

\textbf{Case B.1: } Suppose that $\widetilde{s}\in S$. If $\vert S\vert \leq \vert T\vert$, choose a subset $S\subseteq S'\subseteq E$ such that $\vert S'\vert =\vert T\vert$ and $\widetilde{\nu}_\sfM(S')=\widetilde{\nu}_\sfM(S)$. We may invoke \cite[Theorem 2.1]{Murota_valuationindependentsets}, which tells us that the restriction of $\widetilde{\nu}_\sfM$ to ${E\choose \vert T\vert}$ is a valuated matroid of rank $\vert T\vert$. This means there is $t\in T-S'$ such that 
\begin{equation*}
\widetilde{\nu}_\sfM(S')+\widetilde{\nu}_\sfM(T) \ \ \geq \ \ \widetilde{\nu}_\sfM(S'-\{\widetilde{s}\}\cup\{t\} )+\widetilde{\nu}_\sfM(T-\{t\}\cup\{\widetilde{s}\})
\end{equation*}
Setting $\widetilde{t}=t$, this implies
\begin{equation*}\begin{split}
\nu_{\widetilde{\sfM}}\big(\widetilde{S}\big)+\nu_{\widetilde{\sfM}}\big(\widetilde{T}\big)\ \ &= \ \  \widetilde{\nu}_\sfM(S)+\widetilde{\nu}_\sfM(T) \\
& = \ \  \widetilde{\nu}_\sfM(S')+\widetilde{\nu}_\sfM(T) \\
 &\geq \ \ \widetilde{\nu}_\sfM(S'-\{\widetilde{s}\}\cup\{t\} )+\widetilde{\nu}_\sfM(T-\{t\}\cup\{\widetilde{s}\}) \\
 & \geq \ \   \widetilde{\nu}_\sfM(S-\{\widetilde{s}\}\cup\{t\} )+\widetilde{\nu}_\sfM(T-\{t\}\cup\{\widetilde{s}\}) \\
 &= \ \ \nu_{\widetilde{\sfM}}\big(\widetilde{S}-\{\widetilde{s}\}\cup\{\widetilde{t}\}\big)+\nu_{\widetilde{\sfM}}\big(\widetilde{T}-\{\widetilde{t}\}\cup\{\widetilde{s}\}\big) \ , 
\end{split}\end{equation*}
since $S'-\{\widetilde{s}\}\cup\{t\} \supseteq S-\{\widetilde{s}\}\cup\{t\}$ and thus $\widetilde{\nu}_\sfM(S'-\{\widetilde{s}\}\cup\{t\} )\geq \widetilde{\nu}_\sfM(S-\{\widetilde{s}\}\cup\{t\} )$ by the definition of $\widetilde{\nu}_\sfM$. 

\textbf{Case B.2: } We now prove the case $\vert T\vert < \vert S\vert$ by induction on $k=\vert S\vert -\vert T\vert$, starting with the $k=0$ case that we have taken care of above. Suppose $k\geq 1$. We may find a $t'\in E$ such that $\widetilde{\nu}_\sfM(T) = \widetilde{\nu}_\sfM(T')$ for $T'=T\cup\{t'\}$ by the definition of $\widetilde{\nu}_{\sfM}$. If $t'=\widetilde{s}$, we use $\widetilde{\nu}_\sfM(S)\geq \widetilde{\nu}_\sfM(S-\{\widetilde{s}\})$ and find 
\begin{equation*}\begin{split}
\nu_{\widetilde{\sfM}}\big(\widetilde{S}\big)+\nu_{\widetilde{\sfM}}\big(\widetilde{T}\big)\ \ &= \ \ \widetilde{\nu}_\sfM(S) + \widetilde{\nu}_\sfM(T) \\ 
&= \ \ \widetilde{\nu}_\sfM(S) + \widetilde{\nu}_\sfM(T\cup\{\widetilde{s}\}) \\ &\geq \ \  \widetilde{\nu}_\sfM(S-\{\widetilde{s}\}) + \widetilde{\nu}_\sfM(T\cup\{\widetilde{s}\})\\
& = \ \ \nu_{\widetilde{\sfM}}\big(\widetilde{S}-\{\widetilde{s}\}\cup\{\widetilde{t}\}\big)+\nu_{\widetilde{\sfM}}\big(\widetilde{T}-\{\widetilde{t}\}\cup\{\widetilde{s}\}\big) 
\end{split}\end{equation*}
for an arbitrary $\widetilde{t}\in Q_T-Q_S$ (such a $\widetilde{t}$ exists because $\vert Q_S\vert <\vert Q_T\vert$). 

So now suppose $t'\neq \widetilde{s}$. We may apply the induction hypothesis to $S$ and $T'$ and find a $t\in T'-S$ such that 
\begin{equation*}\begin{split}
 \widetilde{\nu}_\sfM(S)+\widetilde{\nu}_\sfM(T)  
\ \ &= \ \ \widetilde{\nu}_\sfM(S)+\widetilde{\nu}_\sfM(T')\\ 
&\geq \ \ \widetilde{\nu}_\sfM(S-\{\widetilde{s}\}\cup\{t\} )+\widetilde{\nu}_\sfM(T'-\{t\}\cup\{\widetilde{s}\}) 
\end{split}\end{equation*}
If $t\neq t'$, we use $\widetilde{\nu}_\sfM(T'-\{t\}\cup\{\widetilde{s}\})\geq \widetilde{\nu}_\sfM(T-\{t\}\cup\{\widetilde{s}\})$ and, choosing $\widetilde{t}=t$, we find
\begin{equation*}\begin{split}
\nu_{\widetilde{\sfM}}\big(\widetilde{S}\big)+\nu_{\widetilde{\sfM}}\big(\widetilde{T}\big)\ \ & =\ \ \widetilde{\nu}_\sfM(S)+\widetilde{\nu}_\sfM(T) \\ 
&\geq \ \ \widetilde{\nu}_\sfM(S-\{\widetilde{s}\}\cup\{\widetilde{t}\} )+\widetilde{\nu}_\sfM(T-\{\widetilde{t}\}\cup\{\widetilde{s}\})\\
&= \ \ \nu_{\widetilde{\sfM}}\big(\widetilde{S}-\{\widetilde{s}\}\cup\{\widetilde{t}\}\big)+\nu_{\widetilde{\sfM}}\big(\widetilde{T}-\{\widetilde{t}\}\cup\{\widetilde{s}\}\big) 
\end{split}\end{equation*}
which is our claim. When $t=t'$, we choose $\widetilde{t}\in Q_{\widetilde{T}}-Q_{\widetilde{S}}$ using $\vert Q_S\vert <\vert Q_T\vert$ and then we have
\begin{equation*}\begin{split}
\nu_{\widetilde{\sfM}}\big(\widetilde{S}\big)+\nu_{\widetilde{\sfM}}\big(\widetilde{T}\big)\ \ &= \ \ \widetilde{\nu}_\sfM(S)+\widetilde{\nu}_\sfM(T)  \\ 
&= \ \ \widetilde{\nu}_\sfM(S)+\widetilde{\nu}_\sfM(T') \\  
&\geq \ \ \widetilde{\nu}_\sfM(S-\{\widetilde{s}\}\cup\{t\} )+\widetilde{\nu}_\sfM(T'-\{t\}\cup\{\widetilde{s}\})\\  
&\geq \ \ \widetilde{\nu}_\sfM(S-\{\widetilde{s}\} )+\widetilde{\nu}_\sfM(T\cup\{\widetilde{s}\})\\
& = \ \ \nu_{\widetilde{\sfM}}\big(\widetilde{S}-\{\widetilde{s}\}\cup\{\widetilde{t}\}\big)+\nu_{\widetilde{\sfM}}\big(\widetilde{T}-\{\widetilde{t}\}\cup\{\widetilde{s}\}\big)  \ ,
\end{split}\end{equation*}
since $T'=T\cup\{t\}$ and $\widetilde{\nu}_\sfM(S-\{\widetilde{s}\}\cup\{t\} )\geq \widetilde{\nu}_\sfM(S-\{\widetilde{s}\} )$. At this point we have taken care of all possible cases and the proof is complete.
\end{proof}

\begin{remark}
Assume the valuated matroid $\nu_\sfM$ is realizable by a matrix $A\in K^{r\times |E|}$ over a non-Archimedean valued field $K$ with infinite residue field, i.e. we have $\nu_\sfM(S) = -\val(\det[A_S])$ for all $S \in \binom{E}{r}$. Then, the extension $\nu_{\widetilde{\sfM}}$ is realizable by a $r \times (|Q| + |E|)$ matrix $[\,G \mid A\,]$ obtained by attaching an $r \times |Q|$ generic matrix $G$ to the left of $A$ whose entries all have valuation $0$.  
\end{remark}

%%%%%%%%%%%%%%%%%%%%%%%%%%%%%%%%%%%%%%%%%%

For an $M$-convex function $\nu_{\sfP}\colon\Delta_E^r\rightarrow \overline{\Gamma}$ (see  \eqref{eq_MconvexPluecker}),
we analogously define a generic extension $\widetilde{\nu}_\sfP\colon \Delta_{E}^{\leq r}\rightarrow\overline{\Gamma}$ by setting
\begin{equation*}
\widetilde{\nu}_\sfP(\alpha):=\min\big\{\nu_{\sfP}(\beta)\mid \beta\in\Delta_{E}^{r} \textrm{ and } \alpha_e\leq \beta_e \textrm{ for all }e\in E\big\} 
\end{equation*} 
for $\alpha\in \Delta_{E}^{\leq r}$.

The following proposition generalizes Proposition \ref{prop_genericextension} to the setting of $M$-convex functions. This extra step allows us to prove log concavity for valuated polymatroids in Theorem \ref{mainthm_logconcavitypolymatroids} and further results in a slightly stronger log-concave inequality for both valuated matroids and polymatroids in Theorems \ref{mainthm_valMason} and \ref{mainthm_logconcavitypolymatroids}.

\begin{proposition}\label{prop_genericextensionMconvexfunction}
    Let $\nu_{\sfP}\colon \Delta_E^r\rightarrow \overline{\Gamma}$ be an $M$-convex function of rank $r$. 
    Let $Q$ be a finite set disjoint from $E$ and write $\widetilde{E} = Q \sqcup E$. 
     Then the map $\nu_{\widetilde{\sfP}}\colon \Delta_{\widetilde{E}}^r\rightarrow \overline{\Gamma}$ given by 
    \begin{equation*}    \nu_{\widetilde{\sfP}}\big((\alpha_Q,\alpha)\big):=\widetilde{\nu}_{\sfP}(\alpha)
    \end{equation*}
    for any $\widetilde{\alpha}=(\alpha_Q,\alpha) \in \Delta_{\widetilde{E}}^{r}$, where $\alpha \in \Delta_{E}^{\leq r}$ and $\alpha_Q \in \Delta_{Q}^{r-\vert\alpha\vert}$, defines an $M$-convex function. 
\end{proposition}

In order to prove Proposition \ref{prop_genericextensionMconvexfunction} we need to develop a valuated generalization of multisymmetric lifts of polymatroids, as defined, for instance, in \cite[Definition 3.1]{EurLarson}. 

Given a function $\nu\colon \Delta_E^r\rightarrow\overline{\Gamma}$, for any $s\in E$ we write $a_s$ for the minimal non-negative integer such that $a_s\geq \alpha_s$ for all $\alpha\in\Delta_E^r$ with $\nu(\alpha)\neq\infty$. We call the tuple $\mathbf{a}=(a_s)_{s\in E}\in\Z^E$ the \textbf{cage} of $\nu$.  Consider a map $\pi\colon E'\rightarrow E$ of finite sets such that for all $s\in E$ we have $\vert \pi^{-1}(s)\vert=a_s$. We observe that the map
    \begin{equation}\label{eq_multisymmetriclift}
        S'\longmapsto \sum_{s'\in S'}e_{\pi(s')}
    \end{equation}
defines is a natural surjection between ${E'\choose r}\rightarrow \Delta_E^r$. 

Given a function $\nu\colon \Delta_E^r\rightarrow\overline{\Gamma}$ with cage $\mathbf{a}=(a_s)_{s\in E}$, we define the \textbf{multisymmetric lift} of $\nu$ as the function $M_\pi(\nu)\colon{E'\choose r}\rightarrow \overline{\Gamma}$ given by 
\begin{equation}\label{eq_multisymmetriclift}
    M_\pi(\nu)(S'):=\nu\Big(\sum_{s'\in S'} e_{\pi(s')}\Big)
\end{equation}
for $S'\in{E'\choose r}$. The following proposition tells us that the $M$-convexity of $\nu$ is equivalent to the valuated basis exchange property of $M_\pi(\nu)$.

\begin{proposition}\label{prop_multisymmetriclift}
The function $\nu\colon \Delta_E^r\rightarrow\overline{\Gamma}$ is $M$-convex if and only if its multisymmetric lift $M_\pi(\nu)\colon E'\rightarrow\overline{\Gamma}$ defines a valuated matroid.
\end{proposition}

\begin{proof}
    Suppose first that $M_\pi(\nu)$ is a valuated matroid. Let $\alpha,\beta\in\Delta_E^r$ as well as $s\in E$ with $\alpha_s>\beta_s$. We may choose $S',T'\in{E'\choose r}$ such that 
     \begin{equation*}
         \sum_{s'\in S'}e_{\pi(s')}=\alpha \qquad \textrm{ and }\qquad \sum_{t'\in T'}e_{\pi(t')}=\beta  
     \end{equation*}
    with $|S' \cap T' \cap \pi^{-1}(e)| = \min(\alpha_e , \beta_e)$ for all $e \in E$. We then have $S' \cap \pi^{-1}(s) \supsetneq T' \cap \pi^{-1}(s)$, so there exists $\widetilde{s}\in (S'-T') \cap \pi^{-1}(s)$. Since $M_\pi(\nu)$ is a valuated matroid, there is $\widetilde{t}\in T'-S'$ such that 
    \begin{equation*}
        M_\pi(\nu)(S')+M_\pi(\nu)(T') \ \ \geq \ \ M_\pi(\nu)\big(S'-\{\widetilde{s}\}\cup\{\widetilde{t}\}\big)+M_\pi(\nu)\big(T'-\{\widetilde{t}\}\cup\{\widetilde{s}\}\big) \ .
    \end{equation*}
    Note that, in particular, $t:= \pi(\widetilde t)$ satisfies $\beta_t > \alpha_t$. We also have
    \begin{equation*}
    \nu(\alpha)+\nu(\beta) \ \ \geq \ \ \nu(\alpha-e_s+e_t) + \nu(\beta-e_t+e_s)
    \end{equation*}
    by the definition of $M_\pi(\nu)$ in \eqref{eq_multisymmetriclift},
    showing that $\nu$ is $M$-convex.

    Conversely, suppose that $\nu\colon \Delta_E^r\rightarrow\overline{\Gamma}$ is $M$-convex. Let $S',T'\in{E'\choose r}$ as well as $\widetilde{s}\in S'-T'$. Set $s=\pi(\widetilde{s})$ as well as \begin{equation*}
         \alpha = \sum_{s'\in S'}e_{\pi(s')} \qquad \textrm{ and }\qquad \beta = \sum_{t'\in T'}e_{\pi(t')}  \ .
    \end{equation*}
    We are now in one of two cases:

    \textbf{Case A:} When $\alpha_s\leq \beta_s$, we may choose an arbitrary element $\widetilde{t}\in T'-S'$ with $\pi(\widetilde{t})=s$. We then have
    \begin{equation*}\begin{split}
        M_\pi(\nu)(S')+M_\pi(\nu)(T') \ \ &= \ \ \nu(\alpha)+ \nu(\beta) \\ 
        &= \ \ M_\pi(\nu)\big(S'-\{\widetilde{s}\}\cup\{\widetilde{t}\}\big)+M_\pi(\nu)\big(T'-\{\widetilde{t}\}\cup\{\widetilde{s}\}\big) \ .
    \end{split}\end{equation*}

    \textbf{Case B:} When $\alpha_s>\beta_s$, we use that $\nu$ is $M$-convex to find $t\in E$ with $\beta_t>\alpha_t$ such that 
    \begin{equation*}
    \nu(\alpha)+\nu(\beta) \ \ \geq \ \ \nu(\alpha-e_s+e_t) + \nu(\beta-e_t+e_s) \ . 
    \end{equation*}
    Choose an element $\widetilde{t}\in T'-S'$ with $\pi(\widetilde{t})=t$. We then have 
    \begin{equation*}\begin{split}
        M_\pi(\nu)(S')+M_\pi(\nu)(T') \ \ &= \ \ 
        \nu(\alpha) + \nu(\beta) \\ 
        &\geq \ \ \nu(\alpha-e_s+e_t) + \nu(\beta-e_t+e_s) \\
        &= \ \ M_\pi(\nu)\big(S'-\{\widetilde{s}\}\cup\{\widetilde{t}\}\big)+M_\pi(\nu)\big(T'-\{\widetilde{t}\}\cup\{\widetilde{s}\}\big) \ .
    \end{split}\end{equation*}
    Hence $M_\pi(\nu)$ defines a valuated matroid.     
\end{proof}

Armed with this observation, we can now prove Proposition \ref{prop_genericextensionMconvexfunction}.

\begin{proof}[Proof of Proposition \ref{prop_genericextensionMconvexfunction}]
    Let $\nu_{\sfP}\colon \Delta_E^r\rightarrow\overline{\Gamma}$ be an $M$-convex function. By Proposition \ref{prop_multisymmetriclift}, the multisymmetric lift $M_\pi(\nu_\sfP) \colon E'\rightarrow\overline{\Gamma}$ is a valuated matroid of rank $r$. Choose a finite set $Q'$ with $\vert Q'\vert =\vert Q\vert \cdot r$ that is disjoint from $E'$ and a map $\pi_Q\colon Q'\rightarrow Q$ such that $\vert \pi^{-1}_Q(q)\vert=r$ for all $q\in Q$. Set $\widetilde{E'}=Q'\sqcup E'$, $\widetilde{E}=Q\sqcup E$, and define $\widetilde{\pi}\colon \widetilde{E'}\rightarrow \widetilde{E}$ by 
    \begin{equation*}
        \widetilde{\pi}(\widetilde{e})=\begin{cases}
            \pi_Q(\widetilde{e}) & \text{ if } \widetilde{e}\in Q'\ ,\\
            \pi(\widetilde{e}) & \text{ if } \widetilde{e}\in E' \ .
        \end{cases}
    \end{equation*} 
    Now apply Proposition \ref{prop_genericextension} to get that the generic extension of $M_\pi(\nu_\sfP)$ to the ground set $\widetilde{E'}=Q'\sqcup E'$ defines a valuated matroid of rank $r$ on $\widetilde{E'}$. Observe that the generic extension of $M_\pi(\nu_\sfP)$ is the multisymmetric lift $M_{\widetilde{\pi}}(\nu_{\widetilde{\sfP}})$ of $\nu_{\widetilde{\sfP}}$. Therefore, Proposition \ref{prop_multisymmetriclift} implies that $\nu_{\widetilde{\sfP}}$ is $M$-convex. 
\end{proof}
An alternative proof of Proposition \ref{prop_genericextensionMconvexfunction} can be given using generic extensions of \emph{$M^\natural$-functions}, as defined in \cite{MurotaShiouraIII}, by applying \cite[Proposition 3]{MurotaShiouraII} and \cite[Lemma 6]{Murota_strongermultipleexchangeproperty} (also see \cite[Appendix C]{HusicLohoSmithVegh}).

%%%%%%%%%%%%%%%%%%%%%%%%%%%%%%%%%%%%%%%%%%%%%%%%%%%%%%

\section{Valuated bimatroids -- axioms and  examples}

%%%%%%%%%%%%%%%%%%%%%%%%%%%%%%%%%%%%%%%%%%%%%%%%%%%%%%

In this section we recall the definition of valuated bimatroids, as introduced in \cite{Murota_valuatedbimatroids}, and discuss basic properties and examples. 

\begin{definition}\label{def:valuatedmatrixoid}
Let $E,F$ be finite sets and denote by ${E\choose \ast}\times{F\choose \ast}$ the set of pairs $(I,J)$ with $I\subseteq E$, $J\subseteq F$, and $\vert I\vert=\vert J\vert$. A {\bf valuated bimatroid} $\sfA$ over $\overline{\Gamma}$ on the rows $E$ and columns $F$ is given by a {\bf minor valuation function} $\mu_\sfA\colon{E\choose\ast}\times{F\choose \ast}\rightarrow\overline{\Gamma}$ that satisfies the following axioms:
\begin{enumerate}
\item $\mu_\sfA(\emptyset,\emptyset)=0$.
\item For all $(I,J), (I',J') \in {E\choose \ast}\times{F\choose \ast}$ we have:
    \begin{enumerate}[itemindent=0ex, label={(\roman*)}]
    \item If $i' \in I' - I$ at least one of the two statements holds: 
    \begin{itemize}[itemindent=0ex]
        \item there is $i \in I - I'$ such that 
        \begin{equation*}\mu_\sfA(I,J) + \mu_\sfA(I',J') \ \  \geq \ \ \mu_\sfA\big(I-\{i\}\cup\{ i'\},J\big) + \mu_\sfA\big(I'-\{i'\}\cup\{ i\},J'\big) ,\end{equation*} 
        \item there is $j' \in J' - J$ such that 
        \begin{equation*}\mu_\sfA(I,J) + \mu_\sfA(I',J') \ \ \geq \ \  \mu_\sfA\big(I \cup \{i'\},J \cup \{j'\}\big) + \mu_\sfA\big(I' - \{i'\} ,J' - \{j'\} \big).\end{equation*}
    \end{itemize}
    \item If $j \in J - J'$ then at least one of the two statements holds:
    \begin{itemize}[itemindent=0ex]
        \item there is $i \in I - I'$ such that 
        $$\mu_\sfA(I,J) +  \mu_\sfA(I',J') \ \  \geq \ \ \mu_\sfA\big(I - \{i\},J - \{j\}\big) + \mu_\sfA\big(I' \cup \{i\} ,J' \cup \{j\}\big),$$
        \item there is $j' \in J' - J$ such that 
        $$\mu_\sfA(I,J) + \mu_\sfA(I',J') \ \ \geq \ \  \mu_\sfA\big(I ,J-\{j\}\cup\{ j'\}\big) + \mu_\sfA\big(I' ,J'-\{j'\}\cup\{ j\}\big).$$
    \end{itemize}
    \end{enumerate}
\end{enumerate}
\end{definition}

Given a valuated bimatroid $\sfA$ on rows $E$ and columns $F$, we may associate to this datum a function $\nu_{\widehat{\sfA}}\colon {E\sqcup F\choose \vert E\vert}\rightarrow\overline{\Gamma}$ by setting
\begin{equation*}
    \nu_{\widehat{\sfA}}(S):=\mu_\sfA(E-S, S\cap F)
\end{equation*}
for $S\subseteq E\sqcup F$ with $\vert S\vert=\vert E\vert$. 
Using this function, we can characterize valuated bimatroids in more compact terms, as follows:

\begin{proposition}\label{prop_extendedmatroid}
Let $E$ and $F$ be two finite sets. A function $\mu_\sfA\colon{E\choose \ast}\times{F\choose \ast} \rightarrow \overline{\Gamma}$ is the minor valuation function of a bimatroid $\sfA$ if and only if $\nu_{\widehat{\sfA}}$ is the basis valuation function of a valuated matroid $\widehat{\sfA}$ of rank $\vert E\vert$ on $E\sqcup F$ with $\nu_{\widehat{\sfA}}(E)=0$. 
\end{proposition}

We call $\widehat{\sfA}$ the \textbf{valuated matroid associated to $\sfA$}.

\begin{proof}[Proof of Proposition \ref{prop_extendedmatroid}]
Property (1) of valuated bimatroids is expressing the fact that $\nu_{\widehat{\sfA}}(E) = 0$, and property (2) is expressing the Pl\"ucker relations \eqref{eq_Pluecker} defining valuated matroids for the sets $S = (E - I) \sqcup J$ and $T = (E - I') \sqcup J'$, in the two cases depending on whether $s \in I^c -  (I')^c$ or $s \in J - J'$, and also the corresponding two cases for $t \in T$. 
\end{proof}

\begin{example}[Representable bimatroids]\label{example_tropicalization}
Let $E$ and $F$ be finite sets and $A\in K^{E\times F}$ an $E\times F$ matrix over a valued non-Archimedean field $K$. Then the function $\mu_A$ given by $$\mu_A(I,J)=\val\big(\det [A]_{I,J}\big)$$ for $(I, J) \in{E\choose \ast}\times{F\choose \ast}$ naturally defines a valuated bimatroid $A^{\trop}$ on the ground set $E\times F$. 

In order to see this, we consider the matrix
\begin{equation*}
\left[
\begin{array}{c|c}
I_E & A 
\end{array} \right]= 
\left[\begin{array}{ccc|ccc}
1 & & & & &  \\
& \ddots & &  & A & \\
& & 1& & & 
\end{array}
\right]\in K^{E\times (E\sqcup F)}\end{equation*} 
that is extended by the $E\times E$ identity matrix $I_E$ and note that
\begin{equation*}
\det [I_m\vert A]_{E,I^c\sqcup J}=\pm \det[A]_{I,J} \ .
\end{equation*}
The valuations of the maximal minors in $[I_m\vert A]$ form a rank-$|E|$ valuated matroid satisfying $\nu_{\widehat \sfA}(E)=0$, hence the valuations of all minors of $A$ form a valuated bimatroid by Proposition \ref{prop_extendedmatroid}.  

We refer to $A^{\trop}$ as the \textbf{tropicalization} of $A$ and say that a valuated bimatroid arising in this fashion is \textbf{representable} (or \textbf{realizable}) by the matrix $A$ over $K$. 
\end{example}

Given a square matrix 
\begin{equation*}
A=[a_{ij}]_{1\leq i,j\leq n}\in \overline{\Gamma}^{n\times n},
\end{equation*}
its \textbf{tropical determinant} is defined by 
\begin{equation*}
\det A =\min_{\sigma\in S_n} \big\{a_{1\sigma(1)}+\cdots + a_{n\sigma(n)}\big\} \ .
\end{equation*}
We note that, unlike for the classical determinant of a square matrix over a field $K$, this formula is invariant under row and column exchanges. 

\begin{example}[Valuated bimatroids of Stiefel type]\label{example_Stiefelbimatroids}
Let $E$ and $F$ be finite sets. Given a matrix $A\in\overline{\Gamma}^{E\times F}$, the map $\mu_A\colon {E\choose\ast}\times{F\choose \ast}\rightarrow \overline{\Gamma}$ given by 
\begin{equation*}
\mu_A(I,J)=\det [A]_{I,J}
\end{equation*}
for $(I,J)\in {E\choose\ast}\times{F\choose \ast}$ defines a valuated bimatroid $\St(A)$ on the rows $E$ and columns $F$.
This is in fact a realizable valuated bimatroid, as in \cite{FinkRincon_Stiefel}: Choose a generic lift of $A$ to a matrix over some valued non-Archimedean field and apply the reasoning in Example \ref{example_tropicalization}.
We say that valuated bimatroids of the form $\St(A)$ for a matrix $A\in\overline{\Gamma}^{E\times F}$ are of \textbf{Stiefel type}. 
\end{example}

\begin{example}[Transpose]
The \textbf{transpose} $\sfA^T$ of a bimatroid $\sfA$ 
has the minor valuation function $\mu_{\sfA^T}\colon {F\choose \ast}\times{E\choose \ast}\rightarrow\overline{\Gamma}$ given by 
\begin{equation*}
\mu_{\sfA^T}(J,I)=\mu_\sfA(I,J) 
\end{equation*}
for $J\times I\in {F\choose \ast}\times{E\choose \ast}$.
Phrased in terms of valuated matroids on $E\sqcup F$, the transpose $\sfA^T$ of $\sfA$ is exactly the dual valuated matroid. It is an immediate consequence of this definition that $(\sfA^T)^T=\sfA$ for any valuated bimatroid $\sfA$.
\end{example}

%%%%%%%%%%%%%%%%%%%%%%%%%%%%%%%%%%%%%%%%%%%%%%%%%%%%%%

%%%%%%%%%%%%%%%%%%%%%%%%%%%%%%%%%%%%%%%%%%

\section{Lorentzian polynomials and log-concavity}

From now on we assume that $\Gamma$ is an additive subgroup of $\R$, so that it makes sense to consider exponentials of elements in $\Gamma$.
Before proving Theorems \ref{mainthm_valMason}, \ref{mainthm_logconcavitypolymatroids}, and \ref{mainthm_logconcavitybimatroids}, we recall the definition and basic properties of \textbf{Lorentzian polynomials} from \cite{BrandenHuh}. Let $n$ and $d$ be non-negative integers and denote by $H_n^d\subseteq \R[w_1,\ldots, w_n]$ the set of homogeneous polynomials with real coefficients of degree $d$ in $n$ variables $w_1, \ldots, w_n$. We write a polynomial $f(w)\in H_n^d$ in the variables $w=(w_1,\ldots, w_n)$ as 
\begin{equation*}
    f(w)=\sum_{\alpha\in\Delta_n^d}a_\alpha w^\alpha
\end{equation*}
using multi-index notation $w^\alpha=w_1^{\alpha_1}\cdots w_n^{\alpha_n}$ for $\alpha=(\alpha_1,\ldots, \alpha_n)\in \Z_{\geq 0}^n$. In this notional logic, we also set 
\begin{equation*}
\partial^\alpha f=\Big(\frac{\partial}{\partial w_1}\Big)^{\alpha_1}\cdots \Big(\frac{\partial}{\partial w_n}\Big)^{\alpha_n}f \ .
\end{equation*}

Denote by $P_n^d$ the open subset of polynomials in $H_{n}^d$ for which all coefficients $a_\alpha$ are positive. The subspace $\accentset{\circ}{L}_n^d\subseteq P_n^d$ of \textbf{strictly Lorentzian polynomials} is inductively given by $\accentset{\circ}{L}_n^0=P_n^0$, $\accentset{\circ}{L}_n^1=P_n^1$, as well as
\begin{equation*}
    \accentset{\circ}{L}_n^2=\big\{f\in P_n^2 \bigmid \Hess(f) \textrm{ has the Lorentzian signature } (+,-,\ldots,-) \big\},
\end{equation*}
and
\begin{equation*}
\accentset{\circ}{L}_n^d=\big\{f\in P_n^d \bigmid \partial^\alpha f\in \mathring{L}_n^2 \textrm{ for all }\alpha\in\Delta_n^{d-2} \big\} \ .
\end{equation*}
The space $L_n^d$ of \textbf{Lorentzian polynomials} is defined to be the closure of $\accentset{\circ}{L}_n^d$ in $H_n^d$.  

Lorentzian polynomials enjoy several remarkable properties. We recall three of them that will play a role in the proofs of Theorems \ref{mainthm_valMason}, \ref{mainthm_logconcavitypolymatroids}, and \ref{mainthm_logconcavitybimatroids}.
\begin{itemize}
\item Given a fixed constant $0<q\leq 1$, a function $\nu_{\sfP}\colon \Delta_E^r\rightarrow\overline{\Gamma}$ is $M$-convex if and only if the polynomial 
\begin{equation*}
    f_{\sfP}(w)=\sum_{\alpha\in \Delta_E^r}\frac{q^{\nu_\sfP(\alpha)}}{\alpha!}\prod_{s\in E}w_s^{\alpha_s}
\end{equation*}
is Lorentzian (see \cite[Theorem 3.14]{BrandenHuh}). In particular, for a valuated matroid $\sfM$ of rank $r$ on a finite set $E$, the polynomial
\begin{equation*}
f_\sfM(w)=\sum_{S\in {E\choose r}} q^{\nu_{\sfM}(S)} \prod_{s\in S} w_{s} 
\end{equation*}
is Lorentzian. 
\item Given a Lorentzian polynomial $f(w)\in L_m^d$ and a matrix $A\in\R_{\geq 0}^{m\times n}$, the linear coordinate change $f(Aw)$ is also Lorentzian (see \cite[Theorem 2.10]{BrandenHuh}). 
\item A homogeneous polynomial $f(x,y)=\sum_{k=0}^d a_kx^ky^{d-k}$ in two variables with $a_k\geq 0$ is Lorentzian if and only if the sequence $a_k$ is ultra log-concave (see \cite[Example 2.26]{BrandenHuh}). 
\end{itemize}

We now provide the proof of our main results. Since Theorem \ref{mainthm_logconcavitypolymatroids} generalizes Theorem \ref{mainthm_valMason}, we simply provide a proof of the former.

\begin{proof}[Proof of Theorem \ref{mainthm_logconcavitypolymatroids}]
    Let $\sfP=(E,\nu_{\sfP})$ be a valuated polymatroid of rank $r$ on a finite set $E$, given in terms of an $M$-convex function $\nu_P\colon \Delta^r_E\rightarrow\overline{\Gamma}$. Choose an element $0\notin E$ and set $\widetilde{E}=\{0\}\sqcup E$. Using Proposition \ref{prop_genericextensionMconvexfunction} we may form the generic extension $\widetilde{\sfP}$ of $\sfP$ to $\widetilde{E}$.
      
Fix a constant $0<q\leq 1$. By \cite[Theorem 3.14]{BrandenHuh}, the polynomial
\begin{equation*}
f_{\widetilde{\sfP}}(w)=\sum_{\widetilde{\alpha}\in \Delta^r_{\widetilde{E}}} \frac{q^{\nu_{\widetilde{\sfP}}(\widetilde{\alpha})}}{\widetilde{\alpha}!}\prod_{\widetilde{s}\in \widetilde{E}} w_{\widetilde{s}}^{\widetilde{\alpha}_{\widetilde{s}}} 
\end{equation*}
is Lorentzian. Using the decomposition $\widetilde{E}=\{0\}\sqcup E$, we can rewrite
\begin{equation*}
f_{\widetilde{\sfP}}(w)=\sum_{k=0}^r\frac{w_0^{r-k}}{(r-k)!} \left(\sum_{\alpha\in\Delta_E^k}\frac{q^{\widetilde{\nu}_{\sfP}(\alpha)}}{\alpha!} \prod_{s\in E}w_s^{\alpha_s}\right) \ .
\end{equation*}
We may now set $w_0=x$ as well as $w_s=y$ for all $s\in E$. 
The resulting polynomial is
\begin{equation*}\begin{split}
g(x,y)\ \ &= \ \ \sum_{k=0}^r\sum_{\alpha\in\Delta_E^k}  \frac{1}{(r-k)!}q^{\widetilde{\nu}_{\sfP}(\alpha)}\cdot  x^{r-k}\cdot y^{k}\\ &= \ \  \sum_{k=0}^r \frac{1}{(r-k)!} I_k(\sfP)\cdot x^{r-k}\cdot y^k \ .
\end{split}\end{equation*}
The polynomial $g(x,y)$ is Lorentzian by \cite[Theorem 2.10]{BrandenHuh} and, thus, by \cite[Example 2.26]{BrandenHuh}, the sequence $I'_k(\sfP):=\frac{1}{(r-k)!}\cdot I_k(\sfP)$ is ultra log-concave. In other words, we have
\begin{equation*}
\frac{I'_k(\sfP)^2}{{r\choose k}^2} \ \ \geq \ \ \frac{I'_{k+1}(\sfP)}{{r\choose k+1}}\cdot \frac{I'_{k-1}(\sfP)}{{r\choose k-1}}
\end{equation*}
for all $1\leq k<r$. But this is equivalent to \begin{equation*}
I_k(\sfP)^2 \ \ \geq \ \ \frac{k+1}{k} \cdot I_{k+1}(\sfP)\cdot I_{k-1}(\sfP)
\end{equation*} 
for all $1\leq k< r$. 
\end{proof}

\begin{proof}[Proof of Theorem \ref{mainthm_logconcavitybimatroids}]
Let $\sfA$ be a valuated bimatroid on rows $E$ and columns $F$, and set $m=\vert E\vert$. Consider the associated valuated matroid $\widehat{\sfA}$ as in Proposition \ref{prop_extendedmatroid}. Fix a constant $0<q\leq 1$. By \cite[Theorem 3.14]{BrandenHuh} the polynomial
\begin{equation*}
    f_{\widehat{\sfA}}(w)\ \ =\ \ \sum_{S\in {E\sqcup F\choose m}} q^{\nu_{\widehat{\sfA}}(S)}\prod_{s\in S} w_s
\end{equation*}
is Lorentzian. Setting all $w_s=x$ when $s\in E$ and $w_s=y$ when $s\in F$, the resulting polynomial
\begin{equation*}
    g_{\sfA}(x,y)\ \ =\ \ \sum_{k=0}^m R_k(\sfA) x^{m-k} y^{k}
\end{equation*}
remains Lorentzian by \cite[Theorem 2.10]{BrandenHuh}. By \cite[Example 2.26]{BrandenHuh}, this means that the sequence $R_k$ is ultra log-concave, i.e.\ we have
    \begin{equation*}
        \frac{R_k(\sfA)^2}{{m\choose k}^2}\ \ \geq \ \ \frac{R_{k+1}(\sfA)}{{m\choose k+1}}\cdot \frac{R_{k-1}(\sfA)}{{m\choose k-1}}\ .
    \end{equation*}
In order to obtain the general statement, we apply the same argument to $\sfA^T$ and note that the ultra log-concavity for $N=\min\big\{\vert E\vert,\vert F\vert\big\}$ implies the ultra log-concavity for all $N\geq \min\big\{\vert E\vert,\vert F\vert\big\}$.
\end{proof}

%%%%%%%%%%%%%%%%%%%%%%%%%%%%%%%%%%%%%%%%%%%%%%%%%%%%%%

\section{Towards ultra log-concavity: an open question}\label{section_ultralogconcavity}

Let $\sfM=(E,\nu_{\sfM})$ be a valuated matroid of rank $r$ on a finite ground set $E$, and denote by $\widetilde{\nu}_\sfM$ the extension of $\nu_{\sfM}$ to all of ${E\choose \leq r}$ as above. Recall that, given a constant $0<q\leq 1$, we denote
\begin{equation*}
I_k(\sfM):=\sum_{S\in{E\choose k}} q^{\widetilde{\nu}_\sfM(S)} 
\end{equation*}
for $0\leq k\leq r$. Write $N=\vert E\vert$.

\begin{question}
    Is the sequence $I_k(\sfM)$ not only log-concave but also ultra log-concave? I.e., do we have
\begin{equation*}
\Bigg(\frac{I_k(\sfM)}{{N \choose k}}\Bigg)^2 \ \ \geq \ \ \frac{I_{k+1}(\sfM)}{{N \choose k+1}}\cdot \frac{I_{k-1}(\sfM)}{{N\choose k-1}}
\end{equation*}
for all $1\leq k<r$?
\end{question}

In the non-valuated case, this is the strongest version of Mason's conjecture, which is proved in \cite{BrandenHuh, AnariLiuGharanVinzantIII, ChanPak}. Our method of proof in this article via the construction of a generic extension is a generalization of the approach in \cite{RoehrleUlirsch} (similar in spirit to \cite{Lenz_logconcavefvector}) and does not provide us with the desired ultra log-concavity. Unfortunately, a direct generalization of the approach in \cite[Section 4.3]{BrandenHuh} to the setting of valuated matroids via a suitable multivariate Tutte polynomial does not seem to work either.

To provide evidence for this stronger inequality in the valuated case, we have written code which automates the verification of ultra log-concavity for a given valuated matroid and a randomly generated variable $0<q<1$. This code has an in-built numerical accuracy bound which only returns counterexamples with a fixed (but variable) numerical certainty. 

Further, we have implemented methods in \texttt{Oscar} \cite{Oscar} which randomly generate valuated matroids via their basis-valuation maps, a.k.a. Pl\"ucker vectors. We obtain a randomly generated \emph{representable} Pl\"ucker vector by generating a sparse random $d\times n$ matrix over the field of Puiseux series and computing the valuations of its maximal minors. To obtain a non-representable random Pl\"ucker vector, we proceed by taking a random selection of fixed pre-computed non-representable classical matroids, enriching those with the trivial valuation and taking the direct sum of this valuated matroid with a randomly generated representable matroid.

Using our code, we have checked around 2 million random valuated matroids for ultra log-concavity. We sampled among representable valuated matroids of ranks less than or equal to $9$ over ground sets of size less than or equal to $15$, over $10\,000$ for each combination of rank and ground-set size. On non-representable matroids, we have checked over $10\,000$ direct sums of the V\'amos, Fano, and Non-Pappus matroids with random representable matroids of rank $\leq 5$ and ground-set size $\leq 9$ respectively. Again, we evenly distributed among different combinations of rank and ground set size of the representable matroid. Our code does not return a counterexample to ultra log-concavity. This leads us to believe that if a counterexample to the ultra log-concavity exists, it must be either of high rank, high ground set size, or non-realizable in a way we were unable to check.

All code used to check log-concavity and randomly generate Pl\"ucker vectors can be found at \url{https://github.com/VictoriaSchleis/log-concavity/}.

%%%%%%%%%%%%%%%%%%%%%%%%%%%%%%%%%%%%%%%%%%%%%%%%%%%%%%

%%%%%%%%%%%%%%%%%%%%%%%%%%%%%%%%%%%%%%%%%%%%%%%%%%%%%%

 \bibliographystyle{amsalpha}
\bibliography{biblio}{}

\end{document}